\documentclass[11pt,english, reqno]{amsart}
\usepackage{amsmath, amssymb,mathabx, amscd}
\usepackage[pdftex]{graphicx}
\usepackage[shortlabels]{enumitem}
\usepackage{tikz, tikz-cd}
\usepackage[margin=1in]{geometry}
\usepackage{color}
\usepackage{hyperref}
\usepackage{mathrsfs, colonequals, wasysym}
\usepackage{commath}
\usepackage{array}
\usepackage{spalign}
\usepackage{bbold}
\usepackage{algorithm}
\usepackage{algpseudocode}
\usepackage{multirow}

\newcommand{\nn}{\mathbb{N}}

\newcommand{\rr}{\mathbb{R}}

\newcommand{\cH}{\mathcal{H}}
\newcommand{\cM}{\mathcal{M}}

\newtheorem{thm}{Theorem}[section]

\newtheorem{lem}[thm]{Lemma}

\newtheorem{obs}[thm]{Observation}
\newtheorem{rem}[thm]{Remark}
\newtheorem{prop}[thm]{Proposition}
\newtheorem{cor}[thm]{Corollary}

\theoremstyle{definition}
\newtheorem{defin}[thm]{Definition}
\newtheorem{example}[thm]{Example}

\renewcommand{\bar}{\overline}

\renewcommand{\iff}{\Leftrightarrow}

\DeclareMathOperator*{\argmax}{\arg\!\max}
\DeclareMathOperator*{\argmin}{\arg\!\min}

\title{Projected Subgradient Ascent for Convex Maximization}
\author{Pedro Felzenszwalb}
\email{pff@brown.edu}
\address{Brown University}
\author{Heon Lee}
\email{heon\_lee@brown.edu}
\address{Brown University}

\begin{document}

\begin{abstract}
We consider the problem of maximizing a convex function over a closed
convex set in a real Hilbert space. For linear
functions, we show that a single orthogonal projection suffices to obtain
an approximate solution.  For continuous convex functions over convex
sets, we show that projected subgradient ascent converges to a
first-order stationary point when using arbitrarily large step sizes.
Taking the step sizes to infinity leads to
a deterministic variant of the conditional gradient algorithm, and iterated
linear optimization as a special case. 
\end{abstract}
\maketitle





\section{Introduction}

We consider the problem of maximizing a convex function over a
nonempty closed and convex set \(S\) in a real Hilbert space \(\cH\). Of particular interest is the
case of linear optimization: \begin{equation}\label{eq:optimization}
  \max_{x\in S}\;\langle c, x\rangle
\end{equation}
This problem can be solved by various methods, including but not
limited to interior-point methods, the simplex algorithm, proximal
methods, and projected gradient ascent (e.g.,~\cite{nesterov, alizadeh, bertsekas,
  boyd-vandenberghe, dantzig}).

Here we show that \emph{one} orthogonal projection suffices to obtain
an approximate solution to a linear optimization problem.  More
generally, the approach can be viewed as a single step of projected
gradient ascent using a large step size.

Let $P_S(x)$ denote the unique point in $S$ closest to $x$, i.e., the orthogonal projection.  We show
that $P_S(x_0 + \eta c)$ converges to the unique optimal solution of
\eqref{eq:optimization} closest to $x_0$ as $\eta \to \infty$
(Theorem~\ref{thm:x+eta c converges}).  We also give tight bounds on
the objective value of the solution obtained using finite $\eta$
(Lemma~\ref{lemma:bounds}).

Beyond linear objectives, we investigate the use of projected
subgradient ascent for maximizing convex functions.  In the convex
minimization setting, projected subgradient descent requires vanishing
step sizes.  In contrast, we show that for convex maximization,
projected subgradient ascent converges to a first-order stationary
point when using arbitrarily large step sizes
(Theorem~\ref{thm:pga-first-order-stationarity}).  This result holds
for continuous convex functions in infinite-dimensional Hilbert spaces.  Previous work on nonconvex minimization in finite-dimensional spaces (e.g., \cite{attouch-bolte-svaiter}) typically rely on differentiability, Lipschitz continuity of the gradient, and the Kurdyka-Łojasiewicz inequality to ensure convergence.  Our analysis removes these structural requirements by leveraging convexity of the function being maximized.

In practice, the large step size regime of projected subgradient ascent
may lead to faster convergence and yield meaningful behavior.  In the
limit as the step size goes to infinity, projected subgradient ascent leads to the conditional gradient method with unit step size
(Section~\ref{sec:CG-PGA-ILO}).  This limit also generalizes the
iterated linear optimization paradigm introduced
in~\cite{felzenszwalb}.

The idea of using a single projection for linear optimization was
previously considered in \cite{mangasarian} in the context of linear
programming where $S$ is a polytope. The same procedure was
rediscovered in \cite{nurminski} under the assumption of strict
complementarity, and subsequently extended to infinite-dimensional
convex problems satisfying a sharpness property in \cite{bui}.  Here
we consider general convex sets that may have smooth
boundaries.  Such sets arise in a variety of settings including in semidefinite programming.
The idea of using a single projection for linear optimization was also
considered recently, and independently of our work, in
\cite{woodstock}.  Compared to \cite{woodstock}, the bounds we derive
in Section~\ref{sec:reformulation} are tighter, and we prove
convergence to a unique optimum solution besides proving convergence
of the objective value as was done in \cite{woodstock}.


The reduction of linear optimization to orthogonal projection
(Section~\ref{sec:reformulation}) can be used both to understand the
relative complexity of the two operations and to derive new algorithms
for linear optimization using existing algorithms for projection.

\cite{combettes} undertook a complexity analysis and show that on many
domains---the simplex, \(\ell_p\)-balls, nuclear-norm ball, the flow
polytope, Birkhoff polytope, and permutahedron---the best known
algorithms for linear optimization are asymptotically faster than the
best known algorithms for orthogonal projection.  The reduction of
linear optimization to orthogonal projection gives further evidence
that linear optimization over a convex set is no harder than
projecting to the same set.  While \cite{combettes} and
\cite{woodstock} suggest this as a negative result for projection-based methods, we emphasize that the reduction can be used to obtain
positive results, in that efficient algorithms for projections can lead to
efficient algorithms for linear optimization.


\section{Single Projection for Linear Optimization}
\label{sec:reformulation}

Let \(\cH\) be a real Hilbert space and \(S\subseteq\cH\) be a
nonempty closed and convex subset. Throughout the paper, we assume
that the maximum in \eqref{eq:optimization} is attained by at least
one point of \(S\). If \(S\) is weakly compact, then the maximum is
always attained. Define the set of maximizers \[\cM(c) = \argmax_{x\in
  S}\,\langle c, x\rangle.\]

Since \(S\) is closed and convex, and \(\langle c, x\rangle\) is a
linear function of \(x\), the set of maximizers \(\cM(c)\) is
non-empty, closed, and convex.  Let $x_0 \in \cH$.  Denote by \(\|\cdot\|\) the norm
induced by the inner product.  By the Hilbert
Projection Theorem, there is a unique element $x^* \in \cM(c)$ closest to $x_0$,
\[x^* = \argmin_{x\in \cM(c)}\|x-x_0\|.\]

Consider the orthogonal projection map \(P_S: \mathcal H\to S\) taking
a point \(x\in\cH\) to the unique closest point of $x$ in
$S$, \[P_S(x) = \argmin_{y\in S}\|y-x\|.\]

For \(\eta \in\rr\), let \[x^\eta = P_S(x_0+\eta c).\]

Throughout the paper, we let \(\to\) and \(\rightharpoonup\) denote strong and weak convergence respectively. We show that \(x^\eta\to x^*\) as \(\eta\to\infty\) and give
explicit bounds on $\eta$ that guarantee a good approximation.  
This provides a method for approximating \(x^*\) using a single
orthogonal projection.

Figure~\ref{fig:ellipse-example} illustrates the convergence
in the case of an ellipse in the plane with
$x_0$ at the origin.  Since $x^\eta$ is the orthogonal projection of
$\eta c$ we have that $\eta c - x^\eta$ is perpendicular to the
ellipse at $x^\eta$.  Therefore $x^\eta$ maximizes \(\langle \eta c -
x^\eta, x\rangle\).  As $\eta$ grows, $\eta c -x^\eta$ becomes parallel
to $c$, and in the limit $x^\eta$ maximizes \(\langle c, x\rangle\).

\begin{figure}[ht]
    \centering
    \includegraphics[width=0.8\linewidth]{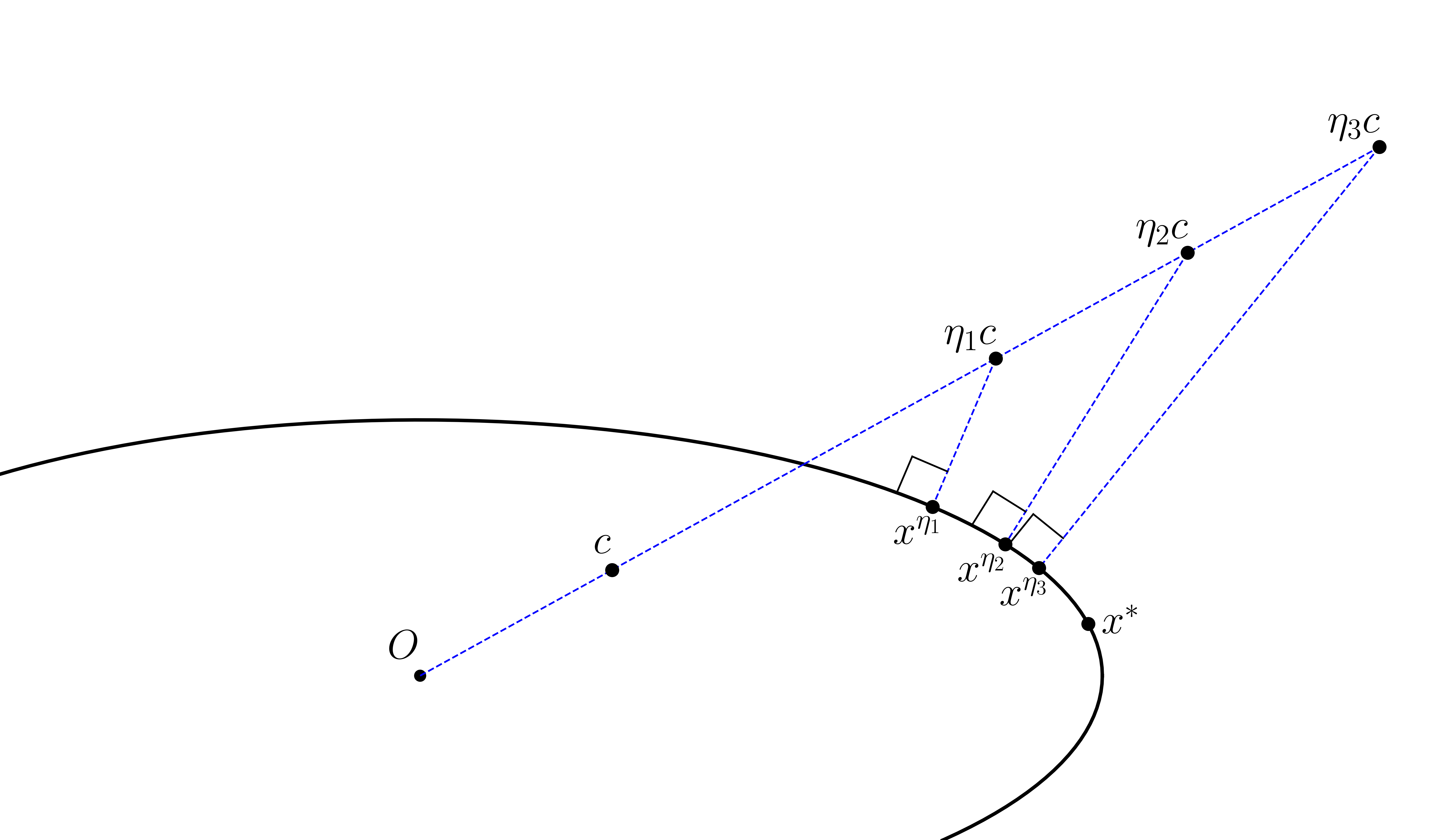}
    \caption{Linear maximization of $\langle c, x \rangle$ via a single projection $x^\eta = P_S(c\eta)$ with $\eta \to \infty$.}
    \label{fig:ellipse-example}
\end{figure}

The use of a single projection for linear optimization was considered 
in \cite{mangasarian} and \cite{nurminski} in the context of linear
programming, where projecting a suitably rescaled cost vector yields
an optimal basic feasible solution. \cite{bui} extended these results to infinite-dimensional spaces under a sharpness assumption. While these work sought an explicit
\(\eta\) that produces an exact solution in the polyhedral setting,  we consider more general (non-polyhedral) closed convex sets in a
Hilbert space and derive bounds on \(\eta\) that ensure an arbitrarily
close approximation to the objective value.  Note that for
smooth sets, no finite $\eta$ achieves optimality.

A result similar to Lemma~\ref{lemma:bounds} appears in
\cite{woodstock}.  However, our lemma provides a more precise
characterization that depends on $x_0$.  This also allows us to prove convergence of
$P_S(x_0+\eta c)$ in Theorem~\ref{thm:x+eta c converges}, including in the
case of general (non-polyhedral) convex sets $S$.

Practical use of $x^\eta$ to approximate $x^*$ requires a
choice for \(\eta\).  The choice should balance the
computational complexity of computing the projection \(x^\eta =
P_S(x_0 + \eta c)\) and the quality of the approximation.  For
some convex sets, projection can be performed efficiently.  Methods
based on alternating projections, such as Dykstra's
algorithm~\cite{dykstra}, can also be used in various settings.

The quality of an approximate solution can be measured in terms of the
difference in objective value $\langle c, x^*\rangle - \langle c,
x^\eta\rangle$.  We first provide a bound on the value of
\(\eta\) sufficient for some desired approximation. Then we demonstrate
convergence of the solution.

We start by recalling some basic results.

\begin{defin}(Normal Cone)
For a point \(x\in S\), the \emph{normal cone} of \(S\) at \(x\) is
the set \[N_S(x) = \{y\in\cH: \langle y, x\rangle \geq \langle y,
z\rangle\,\forall z\in S\}.\]
\end{defin}

The following two lemmas follow from the definition of normal cones
and the optimality condition for projections, i.e., \(\nabla
f(y)^T(z-y)\geq 0\) for all \(z\in S\) where \(y=P_S(x)\) and \(f(y) =
\tfrac12\|y-x\|^2\).

\begin{lem}\label{lem: lin opt normal cone}
Let \(x\in\cH\) and \(y\in S\). Then \[y\in \cM(x)\text{ if and only
  if }x\in N_S(y).\]
\end{lem}

\begin{lem}\label{lem: proj normal cone}
Let \(x\in\cH\) and \(y\in S\). Then \[y=P_S(x)\text{ if and only if
}x-y\in N_S(y).\]
\end{lem}

The next key lemma shows $\langle c,x^\eta \rangle$ approaches $\langle c,x^* \rangle$ as $\eta$ grows.

\begin{lem}
  \label{lemma:bounds}
  Let \(\eta > 0\). Then \[0\leq \langle c, x^*\rangle - \langle c, x^\eta\rangle \leq \frac{\|x^*-x_0\|^2 - \|x^\eta-x_0\|^2}{2\eta}.\]
\end{lem}
\begin{proof}
Since \(x^\eta\) is the closest point in \(S\) from \(x_0 + \eta
c\), \[\|(x_0 + \eta c) - x^\eta\|^2 \leq \|(x_0 + \eta c)- x^*\|^2.\] Expanding and
rearranging the terms, \[2\eta\left(\langle c, x^*\rangle -
\langle c, x^\eta\rangle\right)\leq \|x^*-x_0\|^2 - \|x^\eta-x_0\|^2.\] The
left-hand side is nonnegative because \(x^*\) maximizes \(\langle c,
x\rangle\).  We obtain the desired inequality by dividing both sides by
\(2\eta\).
\end{proof}

When \(S\) is bounded and $x_0 \in S$, we can bound the approximation error
using the diameter of \(S\).  

\begin{obs}
  \label{prop:diam}
Let $x_0 \in S$ and \(\eta > 0\). Then \[0\leq \langle c,x^*\rangle -
\langle c, x^\eta\rangle \leq \frac{\mathrm{diam}(S)^2}{2\eta}.\]
\end{obs}

Observation~\ref{prop:diam} implies we can choose $\eta =
\mathrm{diam}(S)^2/(2\epsilon)$ to ensure $\langle c,x^* \rangle -
\langle c,x^\eta \rangle \le \epsilon$.  The following example
illustrates that this does not mean that $x^\eta$ is close to $x^*$.
In fact $x^\eta$ and $x^*$ may be far in a direction orthogonal to $c$
depending on $x_0$.

\begin{example}\label{ex:bound-tight-square}
Figure~\ref{fig:square-example} shows an example where \(S\subseteq \rr^2\) is a square centered at the origin with
vertices \(\{(\pm 1, \pm1)\}\).   Let $x_0 = O$.  For any \(\eta > 1\),
there exists \(c\) such that \[\|x^* - x^\eta\|^2 \geq
\frac{1}{4}.\] Indeed, let $c = (\tfrac{1}{2\eta}, 1)$. The optimal solution is $x^* =
(1,1)$ while $x^\eta = (\tfrac{1}{2}, 1)$.
\end{example}

\begin{figure}[ht]
    \centering
    \includegraphics[width=0.4\linewidth]{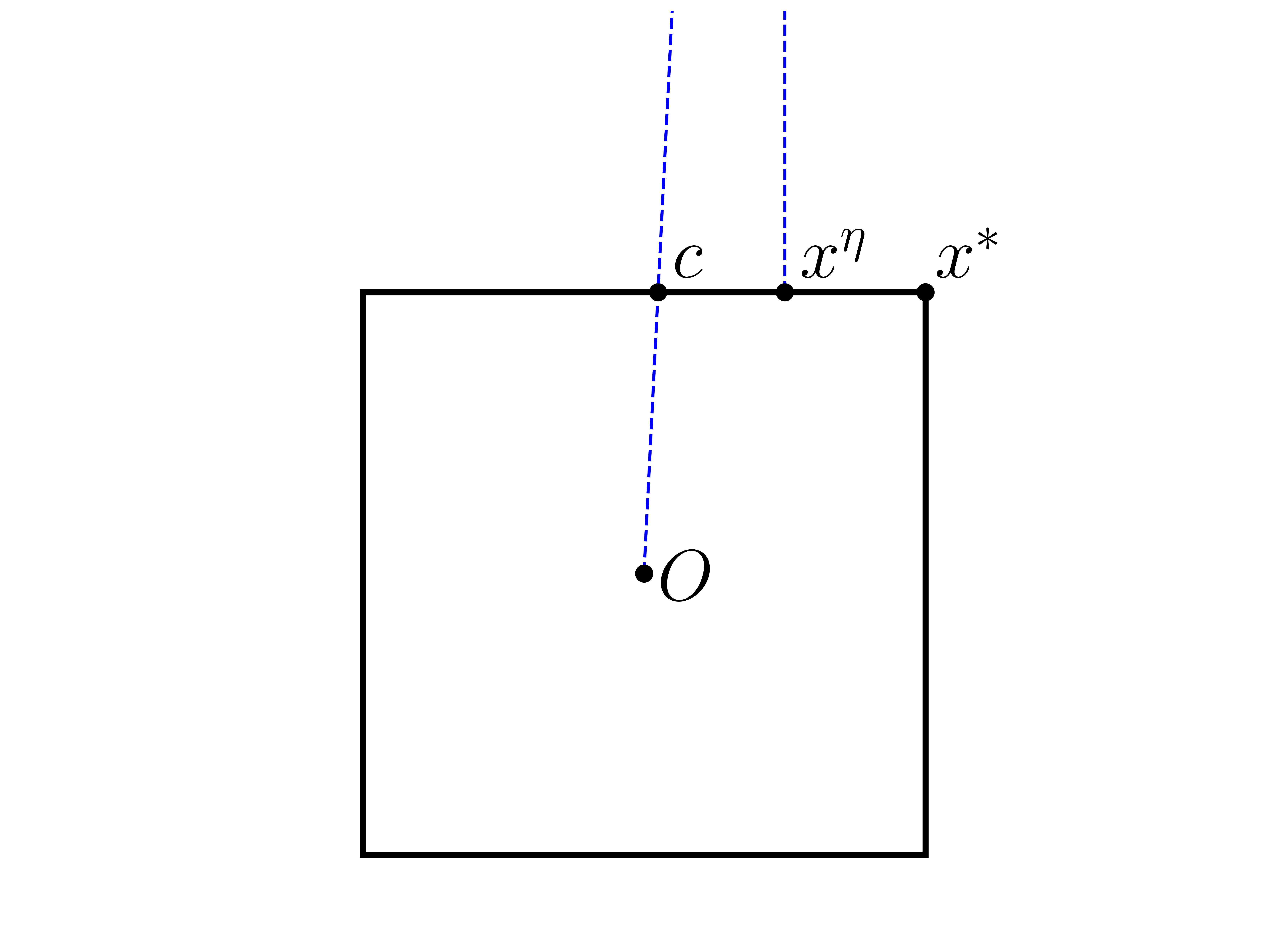}
    \caption{Illustration of Example~\ref{ex:bound-tight-square}.  By
      selecting $\eta$ sufficiently large we can ensure $\langle
      c,x^*\rangle - \langle c,x^\eta \rangle \le \epsilon$
      independent of $c$, while $\|x^*-x^\eta\|$ remains large.}
    \label{fig:square-example}
\end{figure}

Now we show $x^\eta \to x^*$ when $\eta \to \infty$.

Note that both $x^\eta$ and $x^*$ are orthogonal projections, with $x^\eta = P_S(x_0+\eta c)$ and $x^* = P_{\cM(c)}(x_0)$.

\begin{thm}\label{thm:x+eta c converges}
  Let \(S\subseteq \cH\) be a closed and convex set and \(c\in\cH\) such that \(\cM(c)\) is nonempty.  Then \[\lim_{\eta\to\infty}  P_S(x_0+\eta c)  = P_{\cM(c)}(x_0).\]
\end{thm}
\begin{proof}
Lemma~\ref{lemma:bounds} implies that \[\|x^\eta-x_0\| \leq
\|x^*-x_0\|\] for all \(\eta > 0\).  Therefore, \(x^\eta\) lies in a
closed ball of radius \(\|x^*\|\) centered at \(x_0\). Closed balls in Hilbert space are
weakly sequentially compact.  Consider an arbitrary increasing sequence
\(\{\eta_j\}_{j\in\nn}\) such that \(x^{\eta_j} \rightharpoonup \bar x\) for
some \(\bar x\in S\). We show that \(x^{\eta_j}\to x^*\) as \(j\to\infty\).

Lemma~\ref{lemma:bounds} implies that $\bar x \in \cM(c)$.  Lemma~\ref{lemma:bounds} also implies that \(\|\bar x-x_0\| \leq \|x^*-x_0\|\).
Since $x^*$ is the unique element of $\cM(c)$ closest to $x_0$, we conclude $\bar x = x^*$.

Since \(\cH\) is a real Hilbert space, it satisfies the Radon-Riesz property (see, e.g.~\cite{megginson}). To show strong convergence of \(\{x^{\eta_j}\}\), it then remains to show that \(\|x^{\eta_j}\|\to\|x^*\|\). 

Lemma~\ref{lemma:bounds} implies $\|x^\eta-x_0\|^2 \leq \|x^*-x_0\|^2$.  Therefore \(\limsup_{j}\|x^{\eta_j}-x_0\|^2\leq \|x^*-x_0\|^2\). The weak lower semicontinuity of the norm yields \(\|x^*-x_0\| \leq \liminf_{j}\|x^{\eta_j}-x_0\|\). Since the limit inferior is at most the limit superior, combining both inequalities, we obtain \[\|x^*-x_0\|^2 \leq \liminf_j\|x^{\eta_j}-x_0\|^2 \leq \limsup_j\|x^{\eta_j}-x_0\|^2 \leq\|x^*-x_0\|^2.\] Therefore, \(\lim_{j}\|x^{\eta_j}-x_0\|^2 = \|x^{*}-x_0\|^2\). Expanding both sides, we conclude that \(\|x^{\eta_j}\|\to \|x^*\|\).

Since every accumulation point strongly converges to \(x^*\), the full limit holds.
\end{proof}

One naturally expects the linear objective to grow with the scale factor \(\eta\), so that larger \(\eta\) yields progressively better solutions. Indeed, this monotonicity holds.

\begin{prop}
  \label{prop:grow}
If \(\eta_1 < \eta_2\), then \[\langle c, x^{\eta_1}\rangle \leq \langle c, x^{\eta_2}\rangle.\]
\end{prop}
\begin{proof}
  By definition of \(x^\eta\),
  \[\|(x_0 + \eta_1 c)-x^{\eta_1}\|^2 \leq \|(x_0 + \eta_1 c)-x^{\eta_2}\|^2\text{ and }\|(x_0 + \eta_2 c)-x^{\eta_2}\|^2\leq \|(x_0 + \eta_2 c)-x^{\eta_1}\|^2.\]

  Expanding then adding the two inequalities, we obtain
  \[({\eta_2}-\eta_1)\langle c, x^{\eta_1}\rangle \leq ({\eta_2}-\eta_1)\langle c, x^{\eta_2}\rangle.\]
  Since \({\eta_2}-\eta_1 > 0\), we obtain the desired inequality.
\end{proof}

\section{Projected Subgradient Ascent}
\label{sec:PGA}

In this section, we analyze the behavior of projected subgradient ascent
for maximizing a continuous convex function \(f\) over a closed convex set
\(S\subseteq \cH\).  The problem of maximizing a convex function is
NP‐hard (e.g.,~\cite{pardalos-vavasis}).  Therefore, we focus on
obtaining first-order stationary points. 

For convex minimization, projected subgradient descent (PGD) with a
diminishing step size sequence \(\{\eta_k\}_{k\in\nn}\)
satisfying \[\sum_{k=0}^\infty\eta_k=\infty\text{ and }\eta_k \to 0\]
converges to a global minimizer (see, e.g.,~\cite{bertsekas-convex}).  In
contrast, when maximizing a convex function using projected subgradient
ascent (PGA) convergence to a \emph{global optimum} is not guaranteed, regardless
of the choice of step sizes.  Rather than vanishing steps, we focus on
the case of \emph{arbitrarily large step sizes} and show that PGA
\emph{always} converges to a \emph{first-order stationary point} of \(f\).

\begin{defin}(Projected Subgradient Descent)
  Let \(f:\cH\to\rr\) be a convex function, \(S\subseteq \cH\) be a
  nonempty closed convex subset, and \(\{\eta_k\}_{k\in\nn}\) be a
  sequence of step sizes.

  \emph{Projected subgradient descent}
  generates a sequence of iterates \(\{x_k\}_{k\in\nn}\) in
  \(S\), \[x_{k+1} = P_S(x_k - \eta_k g_k)\] where \(g_k\in\partial
  f(x_k)\) is any subgradient of \(f\) at \(x_k\).
\end{defin}

\begin{defin}(Projected Subgradient Ascent)
  Let \(f:\cH\to\rr\) be a convex function, \(S\subseteq \cH\) be a
  nonempty closed convex subset, and \(\{\eta_k\}_{k\in\nn}\) be a
  sequence of step sizes.

  \emph{Projected subgradient ascescent}
  generates a sequence of iterates \(\{x_k\}_{k\in\nn}\) in
  \(S\), \[x_{k+1} = P_S(x_k + \eta_k g_k)\] where \(g_k\in\partial
  f(x_k)\) is any subgradient of \(f\) at \(x_k\).
\end{defin}

Proposition~\ref{prop:increasing value convex} shows that in the
convex maximization regime with PGA, the sequence \(\{f(x_k)\}\) is always
nondecreasing.  When $f$ is bounded, this
implies $\{f(x_k)\}$ converges.  If the step sizes are bounded from
above, Theorem~\ref{thm:pga-accumulation-connected} shows that the
accumulation set of $\{x_k\}$ is connected.  Moreover, if the step
sizes are bounded from below,
Theorem~\ref{thm:pga-first-order-stationarity} shows all of the
accumulation points of $\{x_k\}$ are first-order stationary points of
$f$.

These results illustrate a contrast between convex minimization with
PGD and convex maximization with PGA.  The convex minimization setting
requires vanishing steps for convergence, while in the convex
maximization setting, convergence is guaranteed with large steps.  In
Section~\ref{sec:CG-PGA-ILO}, we consider the limiting behavior when
all of the step sizes go to infinity.



\subsection{Linear Functions} 

Here, we consider the case of maximizing linear functions as it
already captures the intuition that one can take large step sizes.
The case of maximizing convex functions is a natural
generalization.

For a linear objective \(f(x) = \langle c, x\rangle\), maximizing
\(f\) with PGA is equivalent to minimizing \(-f\) with PGD.  In the
finite-dimensional case, classical results using $L$-Lipschitz
gradients guarantee both PGD and PGA converge with any constant
non-zero step size (see, e.g., \cite{bertsekas}).

In the general case of possibly infinite-dimensional spaces, weak
convergence under arbitrary step sizes can be shown using notions from
monotone operator theory (see, e.g., \cite{bauschke-combettes}). We formally state this known result (see, e.g., Theorem 2.5 in~\cite{bauschke-burke-deutsch-hundal-vanderwerf}) as Theorem~\ref{thm:pga-linear-case} to build the intuition for taking large step sizes before we tackle the broader class of general convex functions.


\begin{defin}(Operator)
An \emph{operator} \(T\) on a real Hilbert space \(\cH\) is a
mapping \[T: \cH\rightrightarrows \cH,\] meaning that each \(x\in\cH\)
is assigned a (possibly empty) subset \(T(x) \subseteq \cH\). Its
\emph{domain} is \[\mathrm{dom}\,T = \{x\in\cH: T(x)\neq\emptyset\},\]
its \emph{range} is \[\mathrm{ran}\,T = \{u\in\cH: \exists
x\in\cH\text{ with }u\in T(x)\},\] and its \emph{graph}
is \[\mathrm{gra}\,T=\{(x,u)\in\cH\times\cH: u\in T(x)\}.\]
\end{defin}

\begin{defin}(Monotone Operator)
An operator \(T: \cH\rightrightarrows\cH\) is called \emph{monotone}
if \[\langle x-y, u-v\rangle \geq 0\text{ for all
}(x,u),(y,v)\in\mathrm{gra}\,T.\]
\end{defin}

\begin{defin}(Maximal Monotone Operator)
A monotone operator \(T: \cH\rightrightarrows\cH\) is \emph{maximal
monotone} if its graph cannot be strictly enlarged without losing
monotonicity.
\end{defin}

\begin{defin}(Resolvent)
Let \(T: \cH\rightrightarrows\cH\) be a maximal monotone operator and
let \(\lambda> 0\). The \emph{resolvent} of \(T\) with parameter
\(\lambda\) is the (single-valued) mapping \[J_{\lambda T} =
(\mathrm{Id}+\lambda T)^{-1}.\]
\end{defin}

\begin{thm}\label{thm:pga-linear-case}
Let \(\{\eta_k\}_{k\in\nn}\) be a sequence of step sizes such that \(\sum_{k=0}^\infty\eta_k =
\infty\). Let $\{x_k\}_{k\in\nn}$ be a sequence of iterates generated by PGA with $f(x) = \langle c, x \rangle$.  The sequence \(\{x_k\}_{k\in\nn}\) converges weakly to a point in \(\cM(c)\).
\end{thm}
\begin{proof}
Consider the operator \[A = N_S - c,\] which is maximal monotone
because both \(N_S\) and the constant operator are maximal monotone, and the interior of the domain of the constant operator is the entire Hilbert space
(see, e.g.,~\cite{ryu-boyd-primer}).

Let \(J_{\eta_kA}\) be the resolvent of \(\eta_k A\) where \(\eta_k > 0\).
Then, we observe that \[x_{k+1} = P_S(x_k + \eta_k c) = J_{\eta_k A}(x_k)\] through the following sequence of equivalences:
\begin{align*}
    y = J_{\eta_kA}(x) &\iff y = (\mathrm{Id} + \eta_k(N_S - c))^{-1}(x)\\
    &\iff y = (\mathrm{Id} + N_S - \eta_k c)^{-1}(x)\\
    &\iff x \in y + N_S(y) -\eta_k c\\
    &\iff x + \eta_k c -y\in N_S(y)\\
    &\iff y = P_S(x + \eta_k c). \tag*{(by Lemma \ref{lem: proj normal cone})}
\end{align*}

By assumption, \(\cM(c)\neq\emptyset\).  We observe that the zeros
of \(A\) equals \(\cM(c)\), \[\mathrm{zer}\,A =
\{x\in\cH: 0\in Ax\} = \{x\in\cH: c\in N_S(x)\}= \cM(c).\] Since
\(A\) is a maximally monotone operator such that \(\mathrm{zer}\,
A\neq\emptyset\), and \(x_{k+1} = J_{\eta_k A}(x_k)\), we conclude
that the sequence converges weakly to a point in \(\mathrm{zer}\,A =
\cM(c)\) (see, e.g., Theorem 23.41 of \cite{bauschke-combettes}).
\end{proof}

\begin{rem}
We emphasize the difference between the results in Theorem~\ref{thm:x+eta c
  converges} and Theorem~\ref{thm:pga-linear-case}.  The first
theorem analyzes the behavior of a single step of projected gradient ascent 
\(P_S(x_0 + \eta c)\) as the step size \(\eta\) goes to infinity, proving
strong convergence to the unique solution in \(\cM(c)\) closest to \(x_0\). 
In contrast, the second theorem establishes weak
convergence of projected gradient ascent to some solution 
in \(\cM(c)\) after an infinite number of finite steps.
\end{rem}


\subsection{Convex Functions}

Now, we consider the case of maximizing a continuous convex function with PGA.

In Theorem~\ref{thm:pga-first-order-stationarity}, we establish
first-order stationarity of accumulation points under weak
assumptions. We consider the case of general Hilbert spaces and only require that the
step sizes do not vanish and remain bounded above. Neither differentiability nor
Lipschitz continuity of the gradient is required.

Previous results that apply to the setting of maximizing convex
functions require significantly stronger assumptions.  For
example, the convergence results in \cite{attouch-bolte-svaiter} 
apply in finite-dimensional spaces when \(f\) is differentiable with a
Lipschitz continuous gradient over a nonempty closed set \(S\), and
satisfies the Kurdyka-Łojasiewicz property.

\begin{prop} \label{prop:increasing value convex}
  Let \(f:\cH\to \rr\) be convex, \(S\subseteq\cH\) be a convex set
  and \(\{x_k\}_\nn\) be a sequence of iterates generated by PGA.  Then the
  sequence of values \(\{f(x_k)\}_{k\in\nn}\) is nondecreasing.
\end{prop}
\begin{proof}
Let \[f_k(x) = f(x_k) + \langle x-x_k, g_k\rangle.\] Since $f$ is convex, and $g_k$ is a subgradient of $f$ at $x_k$, we know \(f_k(x)\) lower-bounds \(f(x)\), and they touch at \(x_k\). In other words, \begin{enumerate}[(a)]
    \item \(f_k(x) \leq f(x)\) for all \(x\in\cH\), and
    \item \(f_k(x_k) = f(x_k)\).
\end{enumerate}

We first show that \begin{equation}\label{eq:increasing value for conv
    function} f_k(x_{k+1}) \geq f_k(x_k).
\end{equation}
Note that $x_{k+1} = P_S(x_k + \eta_k g_k)$ and $x_k = P_S(x_k)$.  By Proposition~\ref{prop:grow}
with $\eta_1=0$ and $\eta_2 = \eta_k$ we know that \(\langle g_k,
x_{k+1}\rangle \geq \langle g_k, x_k\rangle\).  This implies \(\langle
x_{k+1}-x_k, g_k\rangle \geq 0\) and now \eqref{eq:increasing value for conv function} follows from the definition of \(f_k\).

The nondecreasing property follows from (a), \eqref{eq:increasing value for conv function}, and (b),
\[f(x_{k+1})\geq f_k(x_{k+1})\geq f_k(x_k) = f(x_k).\]
\end{proof}

\begin{cor}\label{cor:bounded domain convex value converges}
If \(S\) is bounded, the sequence \(\{f(x_k)\}\) converges.
\end{cor}
\begin{proof}
By Proposition~\ref{prop:increasing value convex}, the sequence is
nondecreasing. Because \(S\) is bounded, the sequence is also
bounded. Finally, bounded nondecreasing sequences in \(\rr\) converge.
\end{proof}

\begin{obs}\label{rem:accumulation}
If \(S\) is compact then \(\{x_k\}\) has at least one accumulation point.
\end{obs}

\begin{thm}\label{thm:pga-accumulation-connected}
Let \(f:\cH\to \rr\) be continuous and convex, \(S\subseteq\cH\) be a
compact and convex set, and \(\{x_k\}_\nn\) be a sequence generated by PGA
with step sizes \(\{\eta_k\}_\nn\) where
\(\limsup_{k\to\infty}\eta_k < \infty\). Then the set of accumulation
points of \(\{x_k\}_\nn\) is connected.
\end{thm}
\begin{proof}
  Let \(y_k = x_k + \eta_k g_k,v_k = x_{k+1} - x_k\), and \(\Delta_k =
  f(x_{k+1}) - f(x_k)\). Then \(x_{k+1} = P_S(y_k)\).

We show that \(\|v_k\|\to 0\) as \(k\to\infty\).

By the optimality conditions of projections, \[\langle \eta_k g_k-v_k,
-v_k\rangle = \langle y_k-x_{k+1}, x_k - x_{k+1}\rangle \leq 0,\]

Therefore,
\begin{equation}\label{eq:inequality1} \eta_k\langle
  g_k, v_k\rangle \geq \|v_k\|^2.
\end{equation} 

Next, by the convexity of
\(f\),
\begin{equation}\label{eq:inequality2} f(x_{k+1}) \geq f(x_k) +
  \langle g_k, v_k\rangle.
\end{equation}

Combining inequalities \eqref{eq:inequality1} and
\eqref{eq:inequality2}, we obtain \[\eta_k\Delta_k \geq \|v_k\|^2.\]
Let \(\bar\eta = \lim\sup_{k\to\infty}\eta_k\).  There
exists some \(K\in\nn\) such that for any \(k\geq K\), \(\bar\eta + 1
> \eta_k\). Hence, for \(k\geq K\), \[(\bar\eta+1)\Delta_k >
\eta_k\Delta_k \geq \|v_k\|^2.\] By Corollary~\ref{cor:bounded domain
  convex value converges}, we know 
\((\bar\eta+1)\Delta_k\to0\), implying that \[\|x_{k+1}-x_k\|^2 =
\|v_k\|^2\to 0.\] Therefore, the set of accumulation points is
connected (see, e.g.,~\cite{asic-adamovic}).
\end{proof}

\begin{defin}[First-Order Stationarity]
A point \(x\in S\) is a \emph{first-order stationary point} for the
maximization of \(f\) over \(S\) if there exists a subgradient
\(g\in\partial f(x)\) such that \(\langle g, z-x\rangle \leq 0\) for
all \(z\in S\). Equivalently, if there exists \(g\in\partial f(x)\)
such that \(g\in N_S(x)\).
\end{defin}

\begin{thm}\label{thm:pga-first-order-stationarity}
Let \(f:\cH\to \rr\) be continuous and convex, \(S\subseteq\cH\) be a
compact and convex set, and \(\{x_k\}_\nn\) be a sequence 
generated by PGA with step sizes \(\{\eta_k\}_\nn\) where
\(\limsup_{k\to\infty}\eta_k < \infty\) and
\(\liminf_{k\to\infty}\eta_k > 0\).  Then every accumulation point of
\(\{x_k\}_\nn\) is a first-order stationary point of \(f\).
\end{thm}
\begin{proof}
Choose any accumulation point \(\tilde x\) of \(\{x_k\}\).  We claim that
$\tilde x$ is a first-order stationary point of \(f\). It suffices to
show there exists some subgradient \(g\in\partial f(\tilde x)\)
such that \(g\in N_S(\tilde x)\).

Let \(\underline \eta=\liminf_{k\to\infty}\eta_k>0\).
Then there exists $K\in \nn$ such that for any $k\geq K$, $\underline{\eta}/2 < \eta_k$.

Let \(v_k = x_{k+1} - x_k\).  The proof of
Theorem~\ref{thm:pga-accumulation-connected} showed \(\|v_k\|\to
0\).  Therefore,
\[\lim_{k\to\infty}\frac{\|v_k\|}{\eta_k} \le \lim_{k\to\infty}\frac{2\|v_k\|}{\underline{\eta}} = 0.\] 
Since \(\frac{v_k}{\eta_k}\) converges to 0 in norm,
\begin{equation}\label{eq:vk/nk->0} \lim_{k\to\infty}
  \frac{v_k}{\eta_k} = 0.
\end{equation}

Because \(\tilde x\) is an accumulation point, there exists a
subsequence \(\{x_{k_j}\}\) such that \(x_{k_j}\to \tilde
x\). Additionally, \(\|v_{k_j}\|\to 0\) implies that \(x_{k_j +
  1}\to\tilde x\).

Because \(f:\cH\to\rr\) is convex and continuous, for any \(x\in
\cH\), there exists \(r_x > 0\) and \(L_x\) such that
\(\partial f(B(x, r_x)) \subset B(0, L_x)\) (see, e.g., Proposition
16.17 of~\cite{bauschke-combettes}). The balls form an open cover of
\(S\). Compactness of \(S\) yields a finite subcover, with which we
may find some \(L\) such that \(\{g_{k_j}\}\subseteq
\bigcup_{x\in S}\partial f(x) \subseteq \bar B(0, L).\) Because \(\bar
B(0, L)\) is weakly compact, there exists a further subsequence
\(\{g_{k_{j_\ell}}\}\) such that \(g_{k_{j_\ell}} \rightharpoonup g\)
for some \(g\). Combining with~\eqref{eq:vk/nk->0},
\(g_{k_{j_\ell}}-{v_{k_{j_\ell}}}/{\eta_{k_{j_\ell}}}\rightharpoonup
g\).

Since \(x_{k_{j_\ell}}\to \tilde x, g_{k_{j_\ell}}\rightharpoonup g\),
\(g_{k_{j_\ell}} \in \partial f(x_{k_{j_\ell}})\), and the
subdifferential operator is maximally monotone, we have \(g\in \partial
f(\tilde x)\) (see, e.g., Proposition 20.37
of~\cite{bauschke-combettes}).

Since \(x_{k_{j_\ell}+1}\to \tilde x,
g_{k_{j_\ell}}-{v_{k_{j_\ell}}}/{\eta_{k_{j_\ell}}}\rightharpoonup
g\), \(g_{k_{j_\ell}} - {v_{k_{j_\ell}}}/{\eta_{k_{j_\ell}}}\in
N_S(x_{k_{j_\ell}+1})\) by Lemma~\ref{lem: proj normal cone}, and the
normal cone operator is maximally monotone, we have, \(g\in
N_S(\tilde x)\).

Since \(g\in\partial f(\tilde x)\) and \(g \in N_S(\tilde x)\) we
conclude \(\tilde x\) is a first-order stationary point of \(f\).
\end{proof}

\subsection{Conditional Gradient and Iterated Linear Optimization}
\label{sec:CG-PGA-ILO}

Now we consider the limiting case of PGA when all of the step sizes go
to infinity and relate this limit to the conditional
gradient method and iterated linear optimization.

Consider the limit of the $k$-th PGA iteration as $\eta_k \to \infty$,
\[x_{k+1} = \lim_{\eta_k\to\infty} P_S(x_k + \eta_k g_k).\]

When $f$ is differentiable $g_k = \nabla f(x_k)$ and by
Theorem~\ref{thm:x+eta c converges},
\begin{equation}
  x_{k+1} = P_{\cM(\nabla f(x_k))}(x_k). 
  \label{eq:PGAlimit}
\end{equation}
That is, $x_{k+1}$ is the maximizer of $\langle \nabla f(x_k), x \rangle$ closest to $x_k$.

This limiting behavior of PGA with infinity step size is closely
related to the conditional gradient method, also known as the
Frank-Wolfe algorithm~\cite{frank-wolfe}.  This parallels the convex
minimization setting, where the limit of a PGD step in a polytope is
known to recover a solution of the corresponding linear minimization
problem~\cite{mortagy-gupta-pokutta}.

\begin{defin}(Conditional Gradient/Frank-Wolfe)
Let \(f:\cH\to\rr\) be convex and differentiable, \(S\subseteq \cH\)
be nonempty, closed, and convex, and
\(\{\eta_k\}_{k\in\nn}\subseteq[0,1]\) be a sequence of step sizes,
The CG algorithm generates a sequence of iterates,
 \[x_{k+1} = x_k + \eta_k(z_k-x_k),\] where \[z_k\in\argmax_{z\in
  S}\, \langle \nabla f(x_k), z\rangle.\]
\end{defin}

The unit-step variant of conditional gradient (CGU) sets \(\eta_k=
1\).  This yields
\[x_{k+1} \in \argmax_{z\in S}\,\langle \nabla f(x_k), z\rangle = \cM(\nabla f(x_k)).\]
Now we can see that PGA with infinite step sizes, as defined by
Equation~\eqref{eq:PGAlimit}, is a deterministic variant of CGU, where
in each iteration we select the particular element of $\cM(\nabla
f(x_k))$ that is closest in norm to the last iterate.  When
$|\cM(\nabla f(x_k))| = 1$, such as when $S$ is smooth, the methods
coincide.

Finally, we note that when \(f(x) = \tfrac{1}{2}\|x\|^2\) the CGU
iteration leads to,
\[x_{k+1}\in\argmax_{z\in S}\,\langle x_k, z\rangle,\]
which is exactly the update rule defined by the iterated linear
optimization paradigm described in \cite{felzenszwalb}. That is,
iterated linear optimization is equivalent to CGU with a particular
choice for $f$, and PGA with infinite step size defines a
deterministic variant of both methods.

\section*{Acknowledgements}

This research was partially funded by the Brown University Advanced Undergraduate Research Fellowship under the SPRINT/UTRA program.

\bibliographystyle{abbrv}
\bibliography{refs}

\vspace{1cm}

\end{document}